\documentclass[11pt,reqno,tbtags,draft]{amsart}
\usepackage{amssymb}
\usepackage{url}
\usepackage[square,numbers]{natbib}
\usepackage{bm}
\bibpunct[, ]{[}{]}{;}{n}{,}{,}

\title[On the probability of a Pareto record]
{On the probability of a Pareto record}

\newcommand\urladdrx[1]{{\urladdr{\def~{{\tiny$\sim$}}#1}}}
\author{James Allen Fill}
\address{Department of Applied Mathematics and Statistics,
The Johns Hopkins University,
3400 N.~Charles Street,
Baltimore, MD 21218-2682 USA}
\email{jimfill@jhu.edu}
\urladdrx{http://www.ams.jhu.edu/~fill/}
\thanks{Research for both authors supported by
the Acheson~J.~Duncan Fund for the Advancement of Research in
Statistics.}

\author{Ao Sun}
\address{Department of Applied Mathematics and Statistics, 
The Johns Hopkins University,
3400 N.~Charles Street,
Baltimore, MD 21218-2682 USA}
\email{asun17@jhu.edu}

\makeatletter
\@namedef{subjclassname@2020}{\textup{2020} Mathematics Subject Classification}
\makeatother

\numberwithin{equation}{section}

\allowdisplaybreaks

\theoremstyle{plain}
\newtheorem{theorem}{Theorem}[section]
\newtheorem{lemma}[theorem]{Lemma}
\newtheorem{proposition}[theorem]{Proposition}
\newtheorem{corollary}[theorem]{Corollary}

\theoremstyle{definition}

\newtheorem{definition}[theorem]{Definition}

\newtheorem{remark}[theorem]{Remark}

\newtheorem*{acks}{Acknowledgments}

\theoremstyle{remark}

\newenvironment{romenumerate}[1][-10pt]{
\addtolength{\leftmargini}{#1}\begin{enumerate}
 }{\end{enumerate}}

\newcounter{oldenumi}
{\setcounter{oldenumi}{\value{enumi}}
\begin{romenumerate} \setcounter{enumi}{\value{oldenumi}}}
{\end{romenumerate}}

\newcounter{thmenumerate}

\newcounter{xenumerate}   

\newcommand{\refT}[1]{Theorem~\ref{#1}}
\newcommand{\refC}[1]{Corollary~\ref{#1}}
\newcommand{\refL}[1]{Lemma~\ref{#1}}
\newcommand{\refR}[1]{Remark~\ref{#1}}
\newcommand{\refS}[1]{Section~\ref{#1}}

\newcommand{\refP}[1]{Proposition~\ref{#1}}
\newcommand{\refD}[1]{Definition~\ref{#1}}

\newcommand{\refF}[1]{Figure~\ref{#1}}

\newcommand\marginal[1]{\marginpar{\raggedright\parindent=0pt\tiny #1}}
\newcommand\REM[1]{{\raggedright\texttt{[#1]}\par\marginal{XXX}}}

\begingroup
  \divide\count255 by 60
  \multiply\count255 by -60
  \advance\count255 by \time
  \ifnum \count255 < 10 \xdef\klockan{\the\count1.0\the\count255}
  \else\xdef\klockan{\the\count1.\the\count255}\fi
\endgroup

\def\rompar(#1){\textup(#1\textup)}    

\def\xexp(#1){e^{#1}}

\newcommand\punkt{.\spacefactor=1000}    
\newcommand\iid{i.i.d\punkt}
\newcommand\ie{i.e\punkt}
\newcommand\eg{e.g\punkt}

\newcommand\bbR{\mathbb R}

\newcounter{CC}
\newcounter{cc}

\newcommand\E{\operatorname{\mathbb E{}}}
\renewcommand\P{\operatorname{\mathbb P{}}}

\newcommand\Cov{\operatorname{Cov}}

\newcommand\ga{a} 
\newcommand\gb{b} 

\newcommand\Beta{{\rm B}}
\newcommand\gG{G} 

\newcommand\cL{{\mathcal L}}

\newcommand\bb{\mathbf b}

\newcommand\xx{\mathbf x}
\newcommand\yy{\mathbf y}

\newcommand\XX{\mathbf X}
\newcommand\YY{\mathbf Y}

\newcommand\gGG{\bm{\gG}}

\newcommand\dd{\,\mathrm{d}}
\newcommand\ddx{\mathrm{d}}

\newcommand\tC{{\widetilde C}}
\newcommand\tF{{\widetilde F}}
\newcommand\tG{{\widetilde H}} 

\newcommand\tXX{{\widetilde \XX}}

\newcommand\hF{\widehat F}

\newcommand\hg{\hat{g}}
\newcommand\hH{\widehat H}
\newcommand\hp{\hat{p}}
\newcommand\hW{\widehat W}

\newcommand\hXX{\widehat \XX}

\newcommand\tp{\tilde p}

\newcommand\doi{D_{01}}

\newcommand\RSP{{\rm RP}}
\newcommand\NRSPD{{\rm NRPD}}
\newcommand\PRSPD{{\rm PRPD}}

\newcommand{\ignore}[1]{}

\usepackage{tikz}
\usepackage{amssymb}
\usetikzlibrary{arrows,positioning}
\usepackage[english]{babel}
\usepackage{pgfplotstable}
\usepackage{pgfplots}
\usepackage{adjustbox}
\pgfplotsset{compat=1.3}

\begin{document}

\date{February~20, 2024. Revised May~5, 2024.}

\begin{abstract}
Given a sequence of independent random vectors taking values in ${\mathbb R}^d$ and having common continuous distribution function~$F$, say that the $n^{\rm \scriptsize th}$ observation \emph{sets a (Pareto) record} if it is not dominated (in every coordinate) by any preceding observation.  Let $p_n(F) \equiv p_{n, d}(F)$ denote the probability that the $n^{\rm \scriptsize th}$ observation sets a record.  There are many interesting questions to address concerning $p_n$ and multivariate records more generally, but this short paper focuses on how $p_n$ varies with~$F$, particularly if, under~$F$, the coordinates exhibit negative dependence or positive dependence (rather than independence, a more-studied case).  We introduce new notions of negative and positive dependence ideally suited for such a study, called \emph{negative record-setting probability dependence} (NRPD) and \emph{positive record-setting probability dependence} (PRPD), relate these notions to existing notions of dependence, and for fixed $d \geq 2$ and $n \geq 1$ prove that the image of the mapping $p_n$ on the domain of NRPD (respectively, PRPD) distributions is $[p^*_n, 1]$ 
(resp.,\ $[n^{-1}, p^*_n]$), where $p^*_n$ is the record-setting probability for any continuous~$F$ governing independent coordinates. 
\end{abstract}

\maketitle

\section{Introduction, background, and main results}\label{S:intro}

\subsection{Introduction, notation, and definitions}\label{SS:intro}

We begin with some definitions, including the \refD{D:record} of (multivariate) records as studied in this paper.
For $\xx, \yy \in \bbR^d$, we write $\xx \leq \yy$ or $\yy \geq \xx$ to mean that $x_j \leq y_j$ for $1 \leq j \leq d$, and we write $\xx \prec \yy$ or $\yy \succ \xx$ to mean that $x_j < y_j$ for $1 \leq j \leq d$.  
For $\xx \in \bbR^d$, we use the usual notation $\|\xx\|_1 := \sum_{j = 1}^d |x_j|$.  We use the standard notation $\implies$ for weak convergence of probability measures in Euclidean spaces (or their distribution functions).

Throughout this paper, $\XX^{(1)}, \XX^{(2)}, \dots$ are assumed to be \iid\ (independent and identically distributed) copies of a $d$-dimensional random vector~$\XX$ with distribution function~$F$ and law (or distribution) denoted by $\cL(\XX)$.  \emph{Throughout the paper we restrict attention to continuous~$F$}, mainly to avoid the complicating mathematical nuisance of ties, as explained in \refR{R:continuous}(d).

\begin{remark}
\label{R:continuous}
(a)~As noted by a reviewer of a previous draft, a distribution function~$F$ on $\bbR^d$ is continuous if and only if each of its~$d$ univariate marginals is.  This is easy to prove from the observation in \cite[Section~3 (only in first edition)]{BillingsleyCPM(1968)} that~$F$ corresponding to random vector~$\XX$ is continuous at $\xx \in \bbR^d$ if and only if $F(\xx) = \P(\XX \prec \xx)$.

(b)~Specializing~(a) to $d = 1$, the distribution function of a random variable~$Y$ is continuous if and only if
$\P(Y = y) = 0$ for each $y \in \bbR$.  

(c)~We note in passing, however, that, in contradistinction to~(b), atomlessness of a random vector does \emph{not} imply continuity of the distribution function in dimensions~$2$ and higher; see, \eg,\ \cite[Section~8.5]{Chow(1997)}.

(d)~Combining (a)--(b), it follows that, if the $d$-dimensional random vector~$\XX$ has continuous distribution function~$F$, then almost surely for every $1 \leq j \leq d$ there are no ties among $X^{(1)}_j, X^{(2)}_j, \ldots$.   
\end{remark}

\begin{definition}
\label{D:record}
(a)~For $n \geq 1$, we say that $\XX^{(n)}$ is a \emph{(Pareto) record} (or that it \emph{sets} a record at time~$n$) if $\XX^{(n)} \leq \XX^{(i)}$ fails for all $1 \leq i < n$.

(b)~If $1 \leq k \leq n$,
we say that $\XX^{(k)}$ is a \emph{current record} (or \emph{remaining record}, or \emph{maximum}) at time~$n$ 
if $\XX^{(k)} \leq \XX^{(i)}$ fails for all $1 \leq i \leq n$ with $i \neq k$.

(c)~For $n \geq 1$ 
we let 
$R_n$ denote the number of records $\XX^{(k)}$ with $1 \leq k \leq n$ and let $r_n$ denote the number of remaining records at time~$n$.
\end{definition}

\begin{remark}
\label{R:monotrans}
It is clear from \refD{D:record} that if $\tXX = (g_1(X_1), \ldots, g_d(X_d))$ where $g_1, \ldots, g_d$ are strictly increasing transformations, then the stochastic processes $(R_n)$ and $(r_n)$ are the same for the \iid\ sequence $\tXX^{(1)}, \tXX^{(2)}, \ldots$ as for $\XX^{(1)}, \XX^{(2)}, \ldots$.  Further, since we assume that~$F$ is continuous, it follows from \refR{R:continuous} that the distribution function~$\widetilde{F}$ of~$\tXX$ is also continuous.  
\end{remark}

\begin{remark}
We note that the expected number $r_n$ of maxima at time~$n$ is~$n$ times the probability that $\XX^{(n)}$ sets a record.  Thus our main Theorems~\ref{T:rangeNRSPD}--\ref{T:rangePRSPD} about record-setting probabilities also gives information about the expected number of maxima when \iid\ vectors are sampled.
\end{remark}

Omitting, for now, any dependence on~$F$ or~$d$ from the notation, the probability $p_n$ that $\XX^{(n)}$ sets a record is given by
\begin{align}
p_n &= \int\!\P(\XX \in \ddx \xx) [1 - \P(\XX \geq \xx)]^{n - 1} 
= \int\!\ddx F(\xx) [1 - H(-\xx)]^{n - 1} \nonumber \\
&= \int\!\ddx H(\yy) [1 - H(\yy)]^{n - 1}
= \E[1 - H(-\XX)]^{n - 1}, \label{pformula}  
\end{align}
where~$H$ denotes the distribution function corresponding to $- \XX$.

\begin{remark}
\label{R:manyto1}
For fixed~$d$ and~$n$, the mapping from~$F$ (equivalently, from~$H$) to $p_n$ is many-to-one; recall \refR{R:monotrans}.  In particular, $p_n$ has the same value for all continuous~$F$ such that the coordinates of~$\XX$ are independent.
\end{remark}

\subsection{RP equivalence classes, RP ordering, and the main results} \label{S:RSP}

We see from~\eqref{pformula} that the sequence $(p_n)_{n \geq 1}$ of record-setting probabilities is determined by $\cL(H(-\XX))$.  Conversely, since the distribution of a bounded random variable is determined by its moments, 
$\cL(H(-\XX))$ is determined by $(p_n)_{n \geq 1}$.  We are thus led to define an equivalence relation on (continuous) $d$-dimensional distribution functions~$F$ (for each fixed~$d$) by declaring that 
$F \sim \tF$ if $H(-\XX)$ and $\tG(-\tXX)$ have the same distribution, where $F$, $\widetilde{F}$, $H$, and 
$\tG$ are the distribution functions of $\XX$, $ - \XX$, $\tXX$, and $ - \tXX$, respectively; we 
call this the \emph{record-setting probability} (RP) \emph{equivalence}.

We are now prepared to define a partial order on the RP equivalence classes.

\begin{definition}
\label{D:po}
Let~$C$ and~$\tC$ be RP equivalence classes with (arbitrarily chosen) respective representatives~$F$ and~$\tF$.  We say that \emph{$C \leq \tC$ in the \RSP\ ordering} (or, by abuse of terminology, that $F \leq \tF$ in the 
\RSP\ ordering) if $H(-\XX) \geq \tG(-\tXX)$ stochastically. 
\end{definition}

\begin{remark}
\label{R:Fp}
From~\eqref{pformula} it follow immediately that if $F \leq \tF$ in the \RSP\ ordering, then $p_n \leq \tp_n$ for every~$n$.
\end{remark}

Let~$C^*$ denote the \RSP\ equivalence class corresponding to independent coordinates.  We next introduce new notions of negative dependence and positive dependence; we relate these notions to more standard notions later, in \refS{S:NandP}.  

\begin{definition}
\label{D:NRSPD}
We will say that $F$ is \emph{negatively record-setting probability dependent} (NRPD) if its 
\RSP\ equivalence class~$C$ satisfies $C \geq C^*$ in the \RSP\ ordering.
\end{definition}

\begin{definition}
\label{D:PRSPD}
We will say that $F$ is \emph{positively record-setting probability dependent} (PRPD) if its 
\RSP\ equivalence class~$C$ satisfies $C \leq C^*$ in the \RSP\ ordering.
\end{definition}

\begin{remark}
Thus any~$F$ having independent coordinates is both NRPD and PRPD.
\end{remark}

We can now state our two main results.  For both, let $p_n(F) \equiv p_{n, d}(F)$ denote the probability that the $n^{\rm \scriptsize th}$ observation $\XX^{(n)}$ from the (continuous) distribution~$F$ sets a record, and let $p^*_n$ denote the value when $F \in C^*$.

\begin{theorem}
\label{T:rangeNRSPD}
For each fixed $d \geq 2$ and $n \geq 1$ the image of the mapping $p_n$ on the domain of \NRSPD\ distributions is precisely the interval $[p^*_n, 1]$.
\end{theorem}

\begin{theorem}
\label{T:rangePRSPD}
For each fixed $d \geq 1$ and $n \geq 1$ the image of the mapping $p_n$ on the domain of \PRSPD\ distributions is precisely the interval $[n^{-1}, p^*_n]$.
\end{theorem}

\begin{remark}
(a)~For $d = 1$ and $n \geq 2$ the conclusion of \refT{T:rangeNRSPD} is false, since then $p_{n, 1}(F) \equiv n^{-1}$.  

(b)~For $n = 1$ the results of Theorems \ref{T:rangeNRSPD}--\ref{T:rangePRSPD} are trivial, since we have $p_{1, d}(F) \equiv 1$; so in proving the theorems we may assume $n \geq 2$.
\end{remark} 

\begin{corollary}
\label{C:range}
For fixed $d \geq 2$ and $n \geq 1$ the image of the mapping $p_n$ on the domain of all continuous distributions~$F$ is precisely the interval $[n^{-1}, 1]$, irrespective of~$d$.
\end{corollary}

We outline here the strategy, as illustrated 
in \refF{F:oneslide} and carried out in \refS{S:proofmain}, for proving 
Theorems \ref{T:rangeNRSPD}--\ref{T:rangePRSPD} (and subsequently \refC{C:range}).  
Let $\mathcal{R}_{{\rm N}}$ and 
$\mathcal{R}_{{\rm P}}$ denote the respective images.  It is immediate from our definitions that 
$\mathcal{R}_{{\rm N}} \subseteq [p^*_n, 1]$ and $\mathcal{R}_{{\rm P}} \subseteq [0, p^*_n]$, and by considering just first coordinates (see \refL{L:extremes}) we quickly narrow the latter to $\mathcal{R}_{{\rm P}} \subseteq [n^{-1}, p^*_n]$.  To show the reverse containments, we then fill the interval $[p^*_n, 1]$ with elements of 
$\mathcal{R}_{{\rm N}}$ by choosing distribution functions~$F$ from a certain class of marginalized-Dirichlet distributions and their weak limits, and we fill the interval 
$[n^{-1}, p^*_n]$ with elements of $\mathcal{R}_{{\rm P}}$ by choosing distribution functions~$F$ from a certain class of distributions with positively associated coordinates (more specifically, certain scale mixtures 
of \iid\ Exponential distributions) and their weak limits.

	\begin{figure}
		\centering
		\begin{tikzpicture}[scale=10, >=stealth]
			\draw (0, 0) -- (1, 0); 
\foreach \x/\xtext in {0, 0.1/\frac{1}{n}, 0.4/p^*_n, 1}
\draw (\x,.5pt) -- (\x,-.5pt) node[anchor=north] {$\xtext$}; 

\draw[->, blue] (0.21, 1pt) -- (0.1, 1pt);
\draw[->, blue] (0.29, 1pt) -- (0.4, 1pt);
\node at (0.25, 1pt) [scale=0.8] {\textcolor{blue}{$\mathrm{PA}_a$}};

\draw[->, orange] (0.66, 1pt) -- (0.4, 1pt);
\draw[->, orange] (0.74, 1pt) -- (1, 1pt);
\node at (0.7, 1pt) [scale=0.8] {\textcolor{orange}{$\mathrm{Dir}_a$}};

\node at (0.16, 1.6pt) [scale=0.8] {\textcolor{blue}{$a\to 0$}};
\node at (0.85, 1.6pt) [scale=0.8] {\textcolor{orange}{$a\to 0$}};

\node at (0.34, 1.6pt) [scale=0.8] {\textcolor{blue}{$a\to \infty$}};
\node at (0.54, 1.6pt) [scale=0.8] {\textcolor{orange}{$a\to \infty$}};

\draw[dashed] (0.4, 6pt) node[anchor=south, scale=0.8] {\textcolor{blue}{$a\,\mathrm{PA}_a\overset{\mathcal{L}}{\longrightarrow}$} $\mathrm{Independent~Exp}(1)$ \textcolor{orange}{$\overset{\mathcal{L}}{\longleftarrow}a\,\mathrm{Dir}_a$}} -- (0.4, 0);

\draw[dashed, blue] (0.1, 0.5pt) -- (0.1, 3.5pt) node [anchor = south, scale = 0.8] {\textcolor{blue}{$\substack{a\,\ln{\mathrm{PA}_a}\overset{\mathcal{L}}{\longrightarrow}(Y, \ldots, Y)\\ \mathrm{with}~Y\sim\mathrm{Exp}(1)}$}};

\draw[dashed, orange] (1, 0.5pt) -- (1, 3.5pt) node [anchor = south, scale = 0.8] {$\substack{\mathrm{Dir}(\mathbf{1})\\\mathrm{with}~\|\mathbf{X}\|_1=1}$};


\node at (0.26, -4pt) [scale=0.8]  {$\textcolor{blue}{p_n} = \E{(1-Z_{a, d}^a)^{n-1}}$};

\node at (0.7, -4pt) [scale = 0.8] {$\textcolor{orange}{p_n} = \E{\bigl(1-Z_{a, d}^{d+a-1}\bigr)^{n-1}}$};

\node at (0.53, -8.5pt) [scale = 0.8] {\textcolor{black}{in stochastic ordering as $a\uparrow$}}; 

\node at (0.53, -9.5pt) [scale = 0.8] {\textcolor{black}{(because same is true in LR ordering)}};

\node at (0.05, -6pt) [scale = 0.8] {\textcolor{black}{\fbox{$Z_{a, d}\sim \mathrm{Beta}(a, d)$}}};

\draw[blue] (0.314,-4pt) node[minimum height=11pt,minimum width=15.2pt,draw] {};

\draw[orange] (0.752,-4pt) node[minimum height=13pt,minimum width=28pt,draw] {};

\draw[blue] (0.314, -4.55pt) -- (0.314, -6.5pt) node[anchor=north, scale=0.8] {strictly $\downarrow$};
\draw[orange] (0.752, -4.64pt) -- (0.752, -6.5pt) node[anchor=north, scale=0.8] {strictly $\uparrow$};
		\end{tikzpicture}
\caption{The strategy for proving Theorems \ref{T:rangeNRSPD}--\ref{T:rangePRSPD}; here the random variable ``PA$_a$'' has the PA distribution $\hF_a$ described in \refS{S:PAmon}.}
\label{F:oneslide}
	\end{figure}
	
\subsection{Brief literature review} \label{S:lit}

Let us mention some related literature concerning Pareto records; we continue to assume~$F$ is continuous throughout this review.  The book~\cite{Arnold(1998)} is a standard reference for univariate records (the case $d = 1$).  For multivariate records in the case of independent coordinates, we have already remarked that the 
record-setting probability $p_n = p^*_n$ does not depend on the distributions of the individual coordinates, but other aspects (such as the location of remaining records) do.  Usually, as in~\cite{Bai(2005)} (see also the reference therein), the coordinates are taken to be \iid,\ either Uniform(0, 1) or standard Exponential.  Bai et al.~\cite{Bai(2005)} obtain, for fixed~$d$ and for both $R_n$ and $r_n$, asymptotic expansions as $n \to \infty$ for the expected value and variance and a central limit theorem with a Berry--Esseen bound.  The main contributions of Fill and Naiman~\cite{Fillboundary(2020)} are localization theorems for the \emph{Pareto frontier}---\ie,\ the topological boundary between the record-setting region and its complement when coordinates are \iid\ standard Exponential---and some of those theorems are substantially sharpened in~\cite{Fillsharper(2023)}.  An importance-sampling algorithm for sampling records is presented, and partially analyzed, in~\cite{Fillgenerating(2018)}.  A limiting distribution (again, for fixed~$d$ as $n \to \infty$) is established for the number $r_{n - 1} + 1 - r_n$ of remaining records broken by $\XX^{(n)}$ conditionally given that $\XX^{(n)}$ sets a record, for $d = 2$ in~\cite{Fillbreaking(2021)} and for general~$d$
in~\cite{Fillmulti(2023)}. 

An underlying theme of the present paper is that it is interesting to see how results (for example, concerning asymptotics for moments and distributions for $R_n$ and $r_n$ and localization of the frontier) vary with~$F$.  When~$F$ is the uniform distribution on the $d$-dimensional simplex, Hwang and Tsai~\cite{Hwang(2010)} (see also the references therein, especially Bai et al.~\cite{Bai(2003)}) proceed in a fashion similar to that in~\cite{Bai(2005)} to obtain analogues of the asymptotic results of that earlier paper.  It is worth noting that the computations are more involved in the simplex case than in~\cite{Bai(2005)}, in part because results about $r_n$ no longer translate immediately to results about $R_n$ since the use of 
so-called concomitants (see \refR{R:concomitants}) becomes more involved, and that the results are enormously different; indeed, for example, as noted in the last line of the table on p.~1867 of~\cite{Hwang(2010)}, we have 
$\E r_n \sim (\ln n)^{d - 1} / (d - 1)!$ for independent coordinates while 
$\E r_n \sim \Gamma(1/d)\,n^{(d - 1) / d}$ for uniform sampling from the $d$-simplex.

\subsection{Organization} \label{S:org}

In \refS{S:gen} we record two simple but very useful general observations about the record-breaking probability $p_n$.  In \refS{S:indep} we briefly review the special case of independent coordinates.  In \refS{S:NandP} we relate the notions of NRPD and PRPD to existing notions of negative and positive dependence.  In \refS{S:Dirmon} we introduce and treat a class of examples of NRPD distributions~$F$ closely related to Dirichlet distributions and in \refS{S:PAmon} we introduce and treat a class of PRPD examples that are scale mixtures of \iid\ Exponential coordinates.  Finally, in \refS{S:proofmain} we prove Theorems~\ref{T:rangeNRSPD}--\ref{T:rangePRSPD} and \refC{C:range} and make a few additional remarks concerning the variability of~$p$.

\subsection{Manifesto} \label{S:man}

In light of Theorems~\ref{T:rangeNRSPD}--\ref{T:rangePRSPD} (see also \refF{F:oneslide} and the proof strategy discussed at the end of~\refS{S:RSP}), we regard the marginalized-Dirichlet NRPD distributions and the scale-mixture PRPD distributions we will use to prove the theorems, if not as canonical examples, then at least as standard examples worthy of thorough consideration---in particular, to study how the behaviors of these examples vary with their associated parameter values.  Accordingly, we regard this paper as a pilot study of sorts, and we are presently working to extend (most of) the results of references~\cite{Bai(2005)}, \cite{Hwang(2010)}, \cite{Fillboundary(2020)}--\cite{Fillsharper(2023)}, \cite{Fillgenerating(2018)}, \cite{Fillbreaking(2021)}, and~\cite{Fillmulti(2023)} to these two classes of examples.

\section{The record-breaking probability~$p_n$:\ general information}\label{S:gen}

To carry out our proof strategy for Theorems~\ref{T:rangeNRSPD}--\ref{T:rangePRSPD}, we first need a result that $p_n$ is continuous as a function of $\cL(\XX)$ at any continuous distribution on $\bbR^d$.  For this result (\refP{P:p is continuous}), we do not need to assume that the distributions of the random vectors $\XX(m)$ are continuous.

\begin{proposition}
\label{P:p is continuous}
Fix $d \geq 1$ and $n \geq 1$.  If $\XX(m)$ converges in distribution to~$\XX$ having a continuous distribution, then the corresponding record-setting probabilities satisfy $p_n(m) \to p_n$ as $m \to \infty$.
\end{proposition} 

\begin{proof}
The distribution functions $H_m$ of $- \XX(m)$ and $H$ of $- \XX$ satisfy $H_m\!\implies\!H$.  Moreover, $H$ is continuous, so $H_m(\yy)$ converges to $H(\yy)$ uniformly in~$\yy$ (\cite[Problem~3 in Section~3 (only in first edition)]{BillingsleyCPM(1968)})
and hence (recalling that~$n$ here is fixed) $[1 - H_m(\yy)]^{n - 1}$ converges to $[1 - H(\yy)]^{n - 1}$ uniformly in~$\yy$.  It follows that as $m \to \infty$ we have
\[
p_n(m) = \int\!\ddx H_m(\yy) [1 - H_m(\yy)]^{n - 1} \to \int\!\ddx H(\yy) [1 - H(\yy)]^{n - 1} = p_n.\qedhere
\]
\end{proof}
%

Our next result exhibits the smallest and largest possible values of $p_n$.
\begin{lemma}
\label{L:extremes}
Fix $d \geq 2$ and $n \geq 1$.  We always have $p_n \in [n^{-1}, 1]$, and $p_n = n^{-1}$ and $p_n = 1$ are both possible.
\end{lemma}

\begin{proof}
If $X^{(n)}_1$ sets a one-dimensional record (which has probability $n^{-1}$), then $\XX^{(n)}$ sets a 
$d$-dimensional record.  Thus $p_n \geq n^{-1}$, and equality holds if~$Y$ has any continuous distribution on~$\bbR$ 
and $\XX = (Y, \ldots, Y)$.

At the other extreme, if $d \geq 2$ and (for example) $\XX \geq {\bf 0}$ has any continuous distribution (such as any Dirichlet distribution) satisfying $\|\XX\|_1 = 1$, then $\XX^{(1)}, \XX^{(2)}, \ldots$ form an antichain in the partial order $\leq$ on $\bbR^d$, so $p_n = 1$.   
\end{proof}

For further general information about $p_n$ (in addition to Theorems~\ref{T:rangeNRSPD}--\ref{T:rangePRSPD}, of course), see \refR{R:general}.

\section{Independent coordinates: $p_n^*$} \label{S:indep}

This brief section concerns the case where the coordinates of each observation are independent.  As noted in \refR{R:manyto1}, $p_n$ doesn't otherwise depend on~$F$ in this setting, so we may as well assume that the coordinates are \iid\ Exponential$(1)$.  Then (writing $p^*_n$ for~$p_n$ in this special case)
\begin{align*}
p^*_n 
&= \int\!\P(\XX \in \ddx \xx) [1 - \P(\XX \geq \xx)]^{n - 1} 
= \int_{{\bf 0} \leq \xx \in \bbR^d}\!e^{- \|\xx\|_1} \left( 1 - e^{- \|\xx\|_1} \right)^{n - 1} \dd \xx \\
&= \int_0^{\infty}\!\frac{y^{d - 1}}{(d - 1)!} e^{-y} (1 - e^{-y})^{n - 1} \dd y \\
&= n^{-1} \sum_{j = 1}^n (-1)^{j - 1} \binom{n}{j} j^{- (d - 1)} 
=: n^{-1} \hH^{(d - 1)}_n ( = n^{-1}\mbox{\ when $d = 1$}).
\end{align*}
Alternatively, as pointed out by a reviewer of a previous draft, the same expression can be obtained for $p^*_n$ by applying the principle of inclusion--exclusion to
$\P\left( \bigcup_{i = 1}^{n - 1} \{\XX^{(n)} \leq \XX^{(i)}\} \right)$.

The numbers
\[ 
\hH^{(k)}_n = \sum_{j = 1}^n (-1)^{j - 1} \binom{n}{j} j^{- k} 
\] 
appearing in the expression for $p^*_n$ 
are called \emph{Roman harmonic numbers}, studied in~\cite{Loeb(1989)}, \cite{Roman(1992)}, and~\cite{Sesma(2017)}.
This $\hH^{(k)}_n$ can be written as a positive linear combination of products of generalized harmonic numbers $H^{(r)}_n := \sum_{j = 1}^n j^{-r}$.  In particular, $\hH^{(1)}_n = H^{(1)}_n$.

\begin{remark}
\label{R:concomitants}

(a)~In obvious notation, the numbers $\E r^*_{n, d} = n p^*_{n, d} = \E R^*_{n, d - 1}$ increase strictly in~$n$ for fixed $d \geq 2$, with limit~$\infty$ as $n \to \infty$.  
(Note:\ The equality in distribution of the random variables $r_{n, d}$ and $R_{n, d - 1}$ for general continuous~$F$ follows by standard consideration of concomitants:\ Consider $\XX^{(1)}, \ldots, \XX^{(n)}$ sorted according to the value of the $d^{\rm \scriptsize th}$ coordinate.)
Further, the numbers $p^*_{n, d}$ increase strictly in~$d$ for fixed $n \geq 2$, with limit~$1$ as $d \to \infty$; and decrease strictly in~$n$ for fixed $d \geq 1$, with limit~$0$ as $n \to \infty$.

(b)~For fixed~$d$ we have 
\[
p^*_{n, d} \sim n^{-1} \frac{(\ln n)^{d - 1}}{(d - 1)!}\mbox{\ as $n \to \infty$}.
\] 
Bai et al.~\cite{Bai(2005)} give a more extensive asymptotic expansion. 
\end{remark}

\section{Negative dependence (including NRPD) and \\ positive dependence (including PRPD)} \label{S:NandP}

In this section we review existing notions of negative and positive dependence in Subsections~\ref{S:N}--\ref{S:P} and relate our new notions of NRPD and PRPD to them in Subsection~\ref{S:relate}. 

\subsection{Negative dependence} \label{S:N}

For a discussion of several notions of negative dependence, see~\cite{Joag-Dev(1983)}.  The first two notions in the next definition can be found there, with focus on the first notion (NA); we have created the third by interpolating between the first two.

\begin{definition}
\label{D:negdep}
(a)~Random variables $X_1, \ldots, X_k$ are said to be \emph{negatively associated} (NA) if for every pair of disjoint subsets $A_1$ and $A_2$ of $\{1, \ldots, k\}$ we have
\[
\Cov\{f_1(X_i:i \in A_1),\,f_2(X_j:j \in A_2)\} \leq 0
\]
whenever $f_1$ and $f_2$ are nondecreasing (in each argument) and the covariance is defined.

(b)~Random variables $X_1, \ldots, X_k$ are said to be \emph{negatively upper orthant dependent} (NUOD) if for all real numbers $x_1, \ldots, x_k$ we have
\[
\P(X_i > x_i,\ i = 1, \ldots, k) \leq \prod_{i = 1}^k \P(X_i > x_i).
\]

(c)~We say that random variables $X_1, \ldots, X_k$ are \emph{negatively upper orthant associated} (NUOA) if for every pair of disjoint subsets $A_1$ and $A_2$ of $\{1, \ldots, k\}$ and all real numbers $x_1, \ldots, x_k$ we have
\[
\P(X_i > x_i,\ i = 1, \ldots, k) \leq \P(X_i > x_i,\ i \in A_1)\,\P(X_j > x_j,\ j \in A_2).
\]  
\end{definition}

\begin{remark}
\label{R:ND}
(a)~NA implies NUOA, which implies NUOD.

(b)~Theorem 2.8 in~\cite{Joag-Dev(1983)} gives a way of constructing NA $(X_1, \ldots, X_k)$, namely, if 
$\gG_1, \ldots, \gG_k$ are independent random variables with log-concave densities, then the conditional distribution of $\gGG = (\gG_1, \ldots, \gG_k)$ given $\sum_{j = 1}^k \gG_j$ is NA almost surely.
\end{remark}

\subsection{Positive dependence} \label{S:P}

For a general discussion of various notions of positive dependence focusing on the one in the next definition, see~\cite{Esary(1967)}. 

\begin{definition}
Random variables ${\bf X} = (X_1, \ldots, X_k)$ are said to be \emph{positively associated} (PA) (or simply \emph{associated}) if
\[
\Cov\{f_1({\bf X}),\,f_2({\bf X})\} \geq 0
\]
whenever $f_1$ and $f_2$ are nondecreasing (in each argument) and the covariance is defined.
\end{definition}

\begin{remark}
\label{R:PD}
It is easy to show that if $Z$ and $\gG_1, \ldots, \gG_d$ are independent positive random variables, then
the scale mixture
\[
\XX := (Z \gG_1, \ldots, Z \gG_d)  
\]
is PA.  The proof uses the law of total covariance (conditioning on~$Z$), the fact 
(\cite[Theorem~2.1]{Esary(1967)}) that independent random variables are PA (applied to the conditional covariance), and the fact (\cite[Property P$_3$]{Esary(1967)}, due to Chebyshev) that the set consisting of a single random variable is PA (applied to the covariance of the conditional expectations).
\end{remark}

\subsection{Relation with NRPD and PRPD} \label{S:relate}

Our motivation for regarding NRPD and PRPD of respective Definitions~\ref{D:NRSPD} and~\ref{D:PRSPD} as notions of negative and positive dependence, respectively, is the following observation.
One might suspect that NA implies NRPD and that PA implies PRPD; we are unable to prove 
either implication, but we can prove the weaker results (recall \refR{R:Fp}) that NA implies $p_2 \geq p^*_2$ and PA implies $p_2 \leq p^*_2$. 

To establish the claimed inequalities, in the following proof replace $*$ by $\leq$ if the observations are PA, by $=$ if they have independent coordinates, and by $\geq$ if they are NA.  The claim is that $p_2 * 1 - 2^{-d}$.  To see this, recall~\eqref{pformula}.  We then have
\begin{align*}
p_2
&= \int \P(\XX \in \ddx \xx) [1 - \P(\XX \geq \xx)] = 1 - \int \P(\XX \in \ddx \xx) \P(\XX \geq \xx) \\
&* 1 - \int \P(\XX \in \ddx \xx) \prod_{j = 1}^d \P(X_j \geq x_j) \\
&* 1 - \int \prod_{j = 1}^d [\P(X_j \in \ddx x_j) \P(X_j \geq x_j)] \\
&= 1 - \prod_{j = 1}^d \int [\P(X_j \in \ddx x) \P(X_j \geq x)] \\
&= 1 - \prod_{j = 1}^d \int [\P(X_j \in \ddx x) \P(X^*_j \geq x | X_j = x))] \\
&{} \qquad \qquad \mbox{where $X^*_j$ is an independent copy of $X_j$} \\
&= 1 - \prod_{j = 1}^d \int [\P(X_j \in \ddx x) \P(X^*_j \geq X_j | X_j = x))] \\
&= 1 - \prod_{j = 1}^d \P(X^*_j \geq X_j) = 1 - \prod_{j = 1}^d 2^{-1} = 1 - 2^{-d}.
\end{align*}


\section{Marginalized-Dirichlet distributions $F_a$:\\ strict decreasing monotonicity in the \RSP\ ordering} \label{S:Dirmon}

The following is the usual definition of Dirichlet distribution, where the normalizing constant is the multivariate 
beta-function value $\Beta(\bb) := \frac{\prod_{j = 1}^k \Gamma(b_j)}{\Gamma(\|\bb\|_1)}$.

\begin{definition}
\label{D:Dirichlet}
Let $k \geq 2$ and $\bb = (b_1, \ldots, b_k) \succ {\bf 0}$.
If $\YY = (Y_1, \ldots, Y_k) \succ {\bf 0}$ satisfies $\|\YY\|_1 = 1$ and $(Y_1, \ldots, Y_{k - 1})$ has $(k - 1)$-dimensional density (with respect to Lebesgue measure)
\[
\frac{1}{\Beta(\bb)} y_1^{b_1 - 1} \cdots y_k^{b_k - 1}\,{\bf 1}(y_1 > 0, \ldots, y_k > 0)
\]
with $y_k := 1 - \sum_{j = 1}^{k - 1} y_j$,
then we say that~$\YY$ has the \emph{Dirichlet$(\bb)$ distribution}.
\end{definition}

We will have special interest in taking $\XX = (X_1, \ldots, X_d)$ to be the first~$d$ coordinates of 
$(Y_1, \ldots, Y_{d + 1}) \sim \mbox{Dirichlet$(1, \ldots, 1, a)$}$; we denote the distribution of~$\XX$ in this case by $\mbox{Dir}_a$ and the corresponding distribution function by $F_a$; we refer to the distributions 
$\mbox{Dir}_a$ as \emph{marginalized-Dirichlet distributions}.
We will have occasional interest in taking
\begin{equation}
\label{Dir1}
\XX = (X_1, \ldots, X_d) \sim \mbox{Dirichlet$(1, \ldots, 1) =:$\ Dir$({\bf 1})$}.
\end{equation}

\begin{remark}
(a)~When $a = 1$, the vector~$\XX$ is uniformly distributed in the (open) $d$-dimensional unit simplex
\begin{equation}
\label{simplex}
{\mathcal S}_d := \{ \xx = (x_1, \ldots, x_d):\,x_j > 0\mbox{\ for $j = 1, \ldots, d$ and $\|\xx\|_1 < 1$} \}.  
\end{equation}
This special case is the focus of~\cite{Hwang(2010)}.

(b)~We find explicit computation (exact or asymptotic) of $p_n$ intractable for general Dirichlet distributions.
\end{remark}

Dirichlet distributions exhibit negative dependence among the coordinates according to standard notions~\cite{Joag-Dev(1983)}:

\begin{remark}
\label{R:DirND}
(a)~The distribution $F_a$ is NUOA [recall \refD{D:negdep}(c)] for every $a \in (0, \infty)$, by a simple calculation.

(b)~The distribution $F_a$ is NA if $a \geq 1$.
Indeed, as in~\refD{D:Dirichlet}, let $\bb = (b_1, \ldots, b_k) \succ {\bf 0}$.  The proof [recall \refR{R:ND}(b)] that Dirichlet$(\bb)$ is NA when $b_j \geq 1$ for every~$j$ relies on the following two standard facts:
\begin{enumerate}
\item[(i)]~If $\gG_j \sim \mbox{Gamma$(b_j)$}$ are independent random variables ($j = 1, \ldots, k$), then 
$\|\gGG\|_1 \sim \mbox{Gamma$(\|\bb\|_1)$}$ and
\[
\YY := \left( \frac{\gG_1}{\|\gGG\|_1}, \ldots, \frac{\gG_k}{\|\gGG\|_1} \right) \sim \mbox{Dirichlet$(\bb)$}
\]
are independent.
\item[(ii)]~For any $b \geq 1$, the Gamma$(b)$ density is log-concave.
\end{enumerate}
\end{remark}

Consider $F = F_a$.  The cases $n = 1$ (with $p_n \equiv 1$) and $d = 1$ (where the choice of~$\ga$ is irrelevant) being trivial, in the following monotonicity result we consider only $n \geq 2$ and $d \geq 2$.

\begin{proposition}
\label{P:mono1}
Fix $d \geq 2$ and $n \geq 2$, and let $F = F_a$, \ie,\ $\XX \sim \mbox{\rm Dir}_a$.  Then~$F_a$ is strictly decreasing in the \RSP\ ordering and therefore the probability $p_n(a) := p_n(F_a)$ that $\XX^{(n)}$ sets a record is strictly decreasing in~$\ga$.
\end{proposition}

\begin{proof}
By successive integrations one finds
\[
\P(\XX \geq \xx) = (1 - \|\xx\|_1)^{d + a - 1};
\]
thus $H_a(-\XX) = (1 - \|\XX\|_1)^{d + a - 1}$.  Further, $1 - \|\XX\|_1 \sim \mbox{Beta$(a, d)$}$, so the first assertion is an immediate consequence of \refL{L:Dirmon} below, and the second assertion follows from \refR{R:Fp}.
\end{proof}

Before proceeding to \refL{L:Dirmon}, we remind the reader of the definition of the likelihood ratio partial ordering (specialized to our setting of random variables taking values in the unit interval) and its connection to the well-known stochastic ordering.

\begin{definition}
Given two real-valued random variables~$S$ and~$T$ with respective everywhere strictly positive densities~$f$ and~$g$ with respect to Lebesgue measure on $(0, 1)$, we say that \emph{$S \leq T$ in the likelihood ratio (LR) ordering} if $g(u) / f(u)$ is nondecreasing in $u \in (0, 1)$.
\end{definition}

As noted (for example) in \cite[Section 9.4]{Ross(1996)}, if $S \leq T$ in the LR ordering, then $S \leq T$ stochastically. 

\begin{lemma}
\label{L:Dirmon}
Fix 
a real number $d > 1$, and let $Z_{\ga}$ have the {\rm Beta}$(\ga, d)$ distribution.  Then 
$W_{\ga} := Z_{\ga}^{d + \ga - 1}$ is strictly increasing in the likelihood ratio ordering, and therefore also in the stochastic ordering, as $\ga \in (0, \infty)$ increases. 
\end{lemma}

\begin{proof}
By elementary calculation, $W_{\ga}$ has density $g_{\ga}$ on $(0, 1)$ given by the following expression, with $c_{\ga} := (d + \ga - 1) \Beta(\ga, d)$:
\begin{align*}
g_{\ga}(w) 
&= c_{\ga}^{-1} w^{- (d - 1) / (d + \ga - 1)} \left( 1 - w^{1 / (d + \ga - 1)} \right)^{d - 1} \\ 
&= c_{\ga}^{-1} \left( w^{- 1 / (d + \ga - 1)} - 1 \right)^{d - 1}.
\end{align*}
Letting $0 < \ga < \gb < \infty$ and setting $v := w^{- 1 / (d + \gb - 1)}$ and 
\[t := (d + \gb - 1) / (d + \ga - 1),\]
it then suffices to show for any fixed $t > 1$ that the ratio $(v - 1) / (v^t - 1)$ decreases strictly as~$v$ increases over $(1, \infty)$.

For this, we consider the log-ratio, whose derivative is $h(v) / [(v - 1) (v^t - 1)]$, where
\begin{equation}
\label{hvdef}
h(v) := v^t - 1 - t v^{t - 1} (v - 1);
\end{equation}
so we need only show that $h(v) < 0$ for $v \in (1, \infty)$.  Indeed, since
\[
h'(v) = - t (t - 1) v^{t - 2} (v - 1) < 0
\]
for $v \in (1, \infty)$, we see that $h(v) < h(1) = 0$ for $v \in (1, \infty)$.
\end{proof}   

\section{Positively associated $\hF_a$:\\ strict increasing monotonicity in the \RSP\ ordering} \label{S:PAmon}

Distributions on $\bbR^d$ with positively associated coordinates 
can be constructed in similar fashion to the marginalized-Dirichlet distributions $F_a$ [recall Remarks \ref{R:PD} and \ref{R:DirND}(b)].
Given $a > 0$, let $\hF_a$ denote the PA distribution of
\[
\XX = \left( \frac{\gG_1}{\gG}, \ldots, \frac{\gG_d}{\gG} \right)
\mbox{\ (scale mixture of \iid\ Exponentials)},
\]
where the random variables $\gG, \gG_1, \ldots, \gG_d$ are independent, $\gG \sim \mbox{Gamma$(a)$}$,
and $\gG_j \sim \mbox{Exponential}(1) \equiv \mbox{Gamma$(1)$}$ for $j = 1, \ldots, d$.

\begin{remark}
(a)~Scale mixtures of a finite number of \iid\ Exponential random variables appear in a study of finite versions of 
de Finetti's theorem \cite[(3.11)]{Diaconis(1987)}. 

(b)~We find explicit computation (exact or asymptotic) of $p_n$ intractable for general scale mixtures, let alone for general PA distributions.
\end{remark} 

Similarly to \refP{P:mono1}, in our positive-association example we have the following claim:

\begin{proposition}
\label{P:mono2}
Fix $d \geq 2$ and $n \geq 2$, and let $F = \hF_a$.  Then~$\hF_a$ is strictly increasing in the \RSP\ ordering and therefore the probability $\hp_n(a) := p_n(\hF_a)$ that $\XX^{(n)}$ sets a record is strictly increasing 
in~$\ga$.
\end{proposition}

\begin{proof}
A simple computation for $\xx \geq {\bf 0}$ gives
\[
\P(\XX \geq \xx) = (1 + \|\xx\|_1)^{-a}
\]
and thus $H_a(-\XX) = (1 + \|\XX\|_1)^{-a}$.  Further, $(1 + \|\XX\|_1)^{-1} = \gG / (\gG + \|\gGG\|_1) \sim \mbox{Beta$(a, d)$}$, so the first assertion is an immediate consequence of the following lemma, and the second assertion follows from \refR{R:Fp}.
\end{proof}

\begin{lemma}
Fix 
a real number $d > 1$, and let $Z_{\ga}$ have the {\rm Beta}$(\ga, d)$ distribution.  Then 
$\hW_{\ga} := Z_{\ga}^{\ga}$ is strictly decreasing in the likelihood ratio ordering, and therefore also in the stochastic ordering, as $\ga \in (0, \infty)$ increases. 
\end{lemma}

\begin{proof}
By elementary calculation, $\hW_{\ga}$ has density $\hg_{\ga}$ on $(0, 1)$ given by the following expression, with $c_{\ga} := \ga \Beta(\ga, d)$:
\[
\hg_{\ga}(w) 
= c_{\ga}^{-1} \left( 1 - w^{1 / \ga} \right)^{d - 1}.
\]
Letting $0 < \ga < \gb < \infty$ and setting $v := w^{1 / \ga}$ and $t := \ga / \gb$, 
it then suffices to show for any fixed $t \in (0, 1)$ that the ratio $(1 - v^t) / (1 - v)$ decreases strictly as~$v$ increases over $(0, 1)$.

For this, we consider the log-ratio, whose derivative is 
\[
- h(v) / [(1 - v) (1 - v^t)],
\]
where we again use the definition~\eqref{hvdef}, but now for $v \in (0, 1]$ [and with $t \in (0, 1)$];
so we need only show that $h(v) > 0$ for $v \in (0, 1)$.  Indeed, since
\[
h'(v) = - t (1 - t) v^{t - 2} (1 - v) < 0
\]
for $v \in (0, 1)$, we see that $h(v) > h(1) = 0$ for $v \in (0, 1)$.
\end{proof}

\section{Proofs of Theorems~\ref{T:rangeNRSPD}--\ref{T:rangePRSPD} and \refC{C:range}} 
\label{S:proofmain}

We are now prepared to prove Theorems~\ref{T:rangeNRSPD}--\ref{T:rangePRSPD} and \refC{C:range} according to the outline provided at the end of \refS{S:RSP}; see \refF{F:oneslide}.

\begin{proof}[Proof of \refT{T:rangeNRSPD}]
In light of \refL{L:extremes}, it suffices to show that the image of $p_n$ on the domain of our marginalized-Dirichlet examples $F_a$ is $(p^*_n, 1)$.

We can regard $p_n \equiv p_n(a)$ as a function on the domain $(0, \infty)$ corresponding to our Dirichlet 
index~$\ga$.  Since the density $f_a(x)$ corresponding to $F_a$ at each fixed argument~$x$ is a continuous function of~$\ga$, it follows from Scheff\'{e}'s theorem (\eg,\ \cite[Thm.~16.12]{BillingsleyPM(2012)}) that the corresponding distribution functions $F_a$ are continuous in~$a$ in the topology of weak convergence.  It then follows from 
Propositions~\ref{P:p is continuous} and~\ref{P:mono1} that the image in question is $(p_n(\infty-), p_n(0+))$.

But, as $\ga \to \infty$, it is easy to see that the density of $\ga$ times an observation converges pointwise to the density for independent Exponentials.  By Scheff\'{e}'s theorem and \refP{P:p is continuous}, therefore, $p_n(\infty-) = p^*_n$.

To compute $p_n(0+)$, we first observe that the distribution of an observation $\XX(\ga)$ from $F_a$ is that of
\[
\left( \frac{Y_1}{\|\YY\|_1 + \gG_{\ga}}, \ldots, \frac{Y_d}{\|\YY\|_1 + \gG_{\ga}} \right),
\]
where $Y_1, \dots, Y_d$ are standard Exponential random variables, $\gG_{\ga}$ is distributed 
(unit-scale) Gamma$(\ga)$, and all $d + 1$ random variables are independent.  It follows easily that 
$\XX(\ga)$ converges in distribution to the distribution Dir$({\bf 1})$ mentioned at~\eqref{Dir1}
(for which $p_n = 1$, as mentioned in the proof of \refL{L:extremes})
as $\ga \to 0$.  Thus, by \refP{P:p is continuous}, $p_n(0+) = 1$.   
\end{proof}

\begin{proof}[Proof of \refT{T:rangePRSPD}]
In light of \refL{L:extremes}, it suffices to show that the image on the domain of our PA examples $\hF_a$ is 
$(n^{-1}, p^*_n)$.

In this case we
can regard $p_n \equiv \hp_n(a)$ as a function on the domain $(0, \infty)$ corresponding to our Gamma index parameter~$\ga$.  The value of the density of an observation at a given point $\xx \geq {\bf 0}$ in 
$\bbR^d$ is
\[
\frac{\Gamma(d + \ga)}{\Gamma(\ga)} (1 + \|\xx\|_1)^{- (d + \ga)},
\]
which is a continuous function of $\ga \in (0, \infty)$.  It follows from Scheff\'{e}'s theorem that the corresponding distribution functions $\hF_a$ are continuous in the topology of weak convergence.  It then follows from Propositions~\ref{P:p is continuous} and~\ref{P:mono2} that the image in question is 
$(\hp_n(0+), \hp_n(\infty -))$.

But, as $\ga \to \infty$, it's easy to see that the density of $\ga$ times an observation converges pointwise to the density for independent standard Exponentials.  By Scheff\'{e}'s theorem and \refP{P:p is continuous}, therefore, $\hp_n(\infty -) = p^*_n$.

To compute $\hp_n(0+)$, we can without changing $\hp_n(\ga)$ take an observation $\hXX(\ga)$ to have coordinates that are $\ga$ times the logarithms of those described in our PA example.  According to 
\cite[Theorem~1]{Liu(2017)} and Slutsky's theorem, $\XX(\ga)$ converges in distribution to $(Y, \ldots, Y)$, where~$Y$ is standard Exponential.  By \refP{P:p is continuous}, therefore, $\hp_n(0+) = n^{-1}$.
\end{proof}

\begin{proof}[Proof of \refC{C:range}]
The corollary follows immediately from \refL{L:extremes} and Theorems \ref{T:rangeNRSPD}--\ref{T:rangePRSPD}.  For a considerably simpler proof, one can use the fact (from \refL{L:extremes}) that there are distributions $F_0$ and $F_1$ satisfying $p_n(F_0) = n^{-1}$ and $p_n(F_1) = 1$ for every~$n$.  By defining $F_q$ to be the $(1 - q, q)$ mixture of $F_0$ and $F_1$ for $q \in [0, 1]$, we see from 
\refP{P:p is continuous} (since $F_q$ is clearly continuous in~$q$ in the weak topology) and the intermediate value theorem that the image of $p_n$ on the domain $\{F_q: q \in [0, 1]\}$ contains (and therefore by \refL{L:extremes} equals) $[n^{-1}, 1]$.
\end{proof}

\begin{remark}
\label{R:general}
We have now learned from Theorems~\ref{T:rangeNRSPD}--\ref{T:rangePRSPD} information about how $p_n$ behaves as a function of the (continuous) distribution of~$\XX$.  As a complement, we conclude this paper with general (and rather more mundane) information about how $p_n$ behaves as a function of~$n$ and as a function of~$d$.

(a)~As already noted, from~\eqref{pformula} it is apparent that $p_n$ is nonincreasing in~$n$.  By the dominated convergence theorem,
\[
p_n \downarrow p_{\infty} := \int_{\xx:\,\P(\XX \geq \xx) = 0}\!\P(\XX \in \ddx \xx)
\]
as $n \uparrow \infty$.  For each fixed $d \geq 2$, the image of the mapping $p_{\infty}$ on the domain of all continuous distributions on $\bbR^d$ is the entire interval $[0, 1]$.  To see by example that $q \in [0, 1]$ is in the image, choose the distribution~$F$ of~$X$ to be the $(q, 1- q)$-mixture of any Dirichlet distribution and any of our marginalized-Dirichlet distributions $F_a$.

(b)~To make sense of the question of how $p_n$ varies as a function of $d \in \{1, 2, \ldots\}$, one should specify a \emph{sequence} of distributions, with the $d^{\rm \scriptsize th}$ distribution being over $\bbR^d$.  It is rather obvious that if $d' < d$ and $\XX(d')$ is obtained by selecting any deterministic set of $d'$ coordinates from $\XX(d)$, then $p_n(d') \leq p_n(d)$; in this sense, $p_n(d)$ is nondecreasing in the dimension~$d$.

Fix~$n \geq 1$, and for any specified sequence (in~$d$) of distributions of $\XX(d)$ let 
$p_n(\infty) := \lim_{d \to \infty} p_n(d)$.  The image of the mapping $p_n(\infty)$ on the domain of all sequences of continuous distributions is $[n^{-1}, 1]$.  This follows easily from \refC{C:range}.  Indeed, given 
$q \in [n^{-1}, 1]$, one can choose $\XX = \XX(2) = (X_1, X_2)$ giving $p_n(2) = q$ and then take 
$\XX(d) = (X_1, X_1, \ldots, X_1, X_2)$ for every~$d$.

For all of our standard examples (independent coordinates, our marginalized-Dirichlet distributions $F_a$, and our PA examples $\hF_a$) we have $p_n(\infty) = 1$.  In light of our earlier results, it is sufficient to prove this for the PA examples.  For that, since the Beta$(a, d)$ distributions converge weakly to unit mass at~$0$ as 
$d \to \infty$, it follows from the consequence
\begin{equation}
\label{pnPA}
\hp_n(a) = \E(1 - Z_{a, d}^a)^{n - 1}\mbox{\ where $Z_{a, d} \sim \mbox{Beta$(a, d)$}$}
\end{equation}
of the proof of \refP{P:mono2} that $p_n(\infty) = 1$.  
\end{remark}

\begin{acks}
We thank three anonymous reviewers for helpful comments.
\end{acks}

{\bf Competing interests:\ }The authors declare none.

\bibliography{records}
\bibliographystyle{plain}

\end{document}